\documentclass[12pt]{amsart}
\usepackage{amscd, amsmath, amssymb, amsfonts, amsthm}
\usepackage[margin=3.5 cm]{geometry}
\usepackage{enumitem}
\usepackage{multicol}
\usepackage{hyperref}
\usepackage{tikz}
\usepackage{array}
\usepackage{calc}  
\usepackage{tabularx}
\tikzstyle{vertex}=[circle, draw, inner sep=0pt, minimum size=6pt]

\allowdisplaybreaks
\def \Z {{\mathbb{Z}}}

\newtheorem*{theorem*}{Theorem}
\newtheorem{theorem}{Theorem}[section]
\newtheorem{cor}[theorem]{Corollary}
\newtheorem{conj}[theorem]{Conjecture}
\newtheorem{lemma}[theorem]{Lemma}
\newtheorem{ex}[theorem]{Example}

\newtheorem{definition}[theorem]{Definition}
\newtheorem{qu}[theorem]{Question}
\newtheorem{rem}[theorem]{Remark}
\newtheorem*{ex*}{Example}

\newtheorem{pro}[theorem]{Proposition}
\begin{document}

\title{A note on words having the same image on finite groups}

\author{Shrinit Singh}
\address{Department of Mathematics, Shiv Nadar Institution of Eminence, Greater Noida, Dadri-201314, India }
\email{ss101@snu.edu.in, shrinitsingh@gmail.com}

\subjclass[2020]{20E05, 20E18, 20F10, 20P05}
\keywords{Free Group, Word Map, Profinite Group}

\begin{abstract}
In this work, we explore the following question: If two words in a finitely generated free group have identical images as word maps on every finite group, must they be endomorphic to each other? In this regard, we introduce weak profinite rigidity for words, a parallel to profinite rigidity, as defined in \cite{hanany2020some}. We establish that the powers of primitive words in any finitely generated free group $F_n$ are weakly profinitely rigid. Furthermore, if a word in $F_n$ has the same image on every finite group as a test word in $F_n$, then both words induce the same probability measure on every finite group. We also prove that a test word in $F_n$ is weakly profinitely rigid if and only if it is profinitely rigid. As a consequence, we establish that the powers of surface words, i.e., $(x_1^2\ldots x_n^2)^d$ in $F_n$ and $([x_1,x_2]\ldots [x_{2n-1},x_{2n}])^d$ in $F_{2n}$, for $n \geq 1$ and any integer $d$, are weakly profinitely rigid.

\end{abstract}
\maketitle

\section{Introduction}

Let $F_n = \langle x_1, \ldots, x_n \rangle$ be a free group on $n$ generators. An element $w$ of $F_n$ is said to be a word and has an expression $\prod_{j=1}^s x_{i_j}^{a_j}$, where $i_j \in \{ 1,2, \ldots, n\}, a_j \in \Z\setminus \{0\}$ and for every $j \in \{1,2, \ldots, s-1\},$ $x_{i_j} \neq x_{i_{j+1}}.$ For any group $G$, we use $G^{(n)}$ to denote the direct product of $n$ copies of $G.$ Corresponding to the word $w \in F_n$ as above, we associate a word map $w$ on a group $G$ to be a map from $G^{(n)}$ to $G$ defined as $w(g_1,\ldots,g_n) = \prod_{j=1}^s g_{i_j}^{a_j}.$ The image of the word map $w$ on $G$ is denoted by $G_w.$

Let $G$ be a finite group. A word $w \in F_n$ also induces a probability measure on the finite group $G.$ The probability of an element $g \in G$ induced by the word $w$, denoted by $\mathrm{Pr}_w(g)$, is $$\frac{|(g_1,g_2,\ldots,g_n) \in G^{(n)} \mid w(g_1,g_2,\ldots,g_n) = g|}{|G^{(n)}|}.$$
 We say that $w_1$ and $w_2$ in $F_n$ induce the same probability measure on $G$ if  $\mathrm{Pr}_{w_1}(g) = \mathrm{Pr}_{w_2}(g)$ for every $g \in G$. Otherwise, we say that they induce different probability measures. Two words $w_1$ and $w_2$ in $F_n$ are called {\it automorphic} if there exists an automorphism $\phi$ of the free group $F_n$ such that $\phi(w_1) = w_2.$ It is easy to note that one cannot distinguish automorphic words just based on their induced probability measures on finite groups. Following \cite{hanany2020some}, we say that a word $w \in F_n$ is {\it profinitely rigid} if every word $u\in F_n$ that induces the same probability measure as $w$ on every finite group is automorphic to $w$.

 Finite groups have been studied using word maps \cite{Bray,cockeho,Segal}. On the other hand, one can talk about the relations among words if the corresponding word maps on every finite group behave in a similar manner. In this direction, the following conjecture by Amit--Vishne and independently by Shalev is well known.  

\begin{conj}\cite{amit2011characters,shalev2013some} \label{shalev}
Every word in $F_n$, for $n \geq 2$, is profinitely rigid.
\end{conj}
 Although the conjecture is still open in general for words in $F_n$, $n \geq 2$, there have been recent developments in its favour. In 2015, Puder and Parzanchevski \cite{puder2015measure} proved that primitive words in any free group are profinitely rigid. In 2020, Hanany, Meiri, and Puder \cite{hanany2020some} proved that $x_1^d$ and $[x_1, x_2]^d$, where $d \in \Z$, are profinitely rigid in every free group $F_n$ for $n \geq 2$. In 2021, Wilton \cite{wilton2021profinite} proved that surface words,  i.e., $x_1^2\ldots x_n^2$ in $F_n$ and $[x_1,x_2]\ldots [x_{2n-1},x_{2n}]$ in $F_{2n}$ for $n \geq 1$, are profinitely rigid.     

In this article, we examine the image of a word map on every finite group while disregarding the probability measure induced by these words. For example,
let $w_1 = x_1$ and $w_2 = x_1^2x_2x_1^{-1}x_2^{-1}$ be two words in $F_2$. They have the same image in every group, but they induce different probability measures on the symmetric group $S_3$. By \cite{puder2015measure}, we get that $w_1$ and $w_2$ are not automorphic. The primary objective of this article is to initiate a study on the relationship among words that have the same image on every finite group.

The simplest case is to establish a relationship among words whose word maps are surjective on every finite group. We call a word $w \in F_n$ \textit{surjective} if $G_w = G$ for every group $G$. 

In this regard, we have the following classification of surjective words in $F_n$.
\begin{theorem}\cite[Theorem~3.1.1]{segal2009words}\label{segal}
    Let $w \in F_n$ and $F_n'$ be the commutator subgroup of $F_n$. Then the following are equivalent:
    \begin{enumerate}
        \item $w$ is surjective;
        \item there exist integers $k_1, k_2, \ldots k_n$ with $\gcd(k_1, k_2, \ldots, k_n) = 1$ such that $$w \in {x_1}^{k_1}{x_2}^{k_2}\ldots {x_n}^{k_n}{F_n}';$$
        \item $G_w = G$ for every finite group $G$;
        \item $G_w = G$ for every finite cyclic group $G$.
    \end{enumerate}
\end{theorem}
Segal has mentioned the equivalence of the first two statements. We have extended this result by adding the last two statements. The equivalence of the last three statements follows easily and is provided in the final section for completeness (see Theorem \ref{prim}).
\begin{definition}\label{def_end}
    Two elements $g$ and $h$ in a group $G$ are endomorphic equivalent if there exist  endomorphisms $\phi_1$ and $\phi_2$ (not necessarily distinct) of $G$ such that $\phi_1(g) = h$ and $\phi_2(h) = g.$ 
\end{definition}

Using Theorem \ref{segal}, we observe that any two surjective words are endomorphic equivalent as follows. Let $w = {x_1}^{k_1}{x_2}^{k_2}\ldots {x_n}^{k_n}c \in F_n$, where $c \in F_n'$ and $\gcd(k_1,k_2\ldots,k_n)=1$, be a surjective word. For any surjective word $u \in F_n$, we define an endomorphism $\phi$ of $F_n$ by mapping $x_i$ to ${u}^{e_i}$ for $i \in \{1,2,\ldots, n\}$, where $e_i \in \mathbb{Z}$ such that $\sum_{i=1}^ne_ik_i = 1$. It is straightforward to verify that $\phi(w) = u$. Similarly, there exists an endomorphism $\psi$ of $F_n$ such that $\psi(u) = w$.

Building on this observation, we present the following lemma, whose proof will be provided in the next section.

\begin{lemma}\label{endoequi}
    Two words $w_1$ and $w_2$ in a free group $F_n$ are endomorphic equivalent if and only if they have the same image as word maps on every group. 
\end{lemma}

In light of Lemma \ref{endoequi} and Theorem \ref{segal}, we raise the natural question: For $w_1, w_2 \in F_n$, is it necessarily true that $w_1$ and $w_2$ are endomorphic equivalent whenever they have the same image on every finite group? In this regard, we define weak profinite rigidity as a natural analogue of profinite rigidity.

\begin{definition}\label{wpr1}
    A word $w \in F_n$ is said to be weakly profinitely rigid if every word $u\in F_n$ that has the same image as $w$ on every finite group is endomorphic equivalent to $w$.
\end{definition}
In the free group $F_1 = \langle x_1 \rangle$, any two words $x_1^l$ and $x_1^k$ will have the same image on every finite group if and only if $l = \pm k$. Clearly, $x_1^l$ and $x_1^{-l}$ are automorphic. Hence, every word in $F_1$ is weakly profinitely rigid. Using Theorem \ref{segal} and the observation that two surjective words are endomorphic equivalent, we have also established weak profinite rigidity of surjective words in a finitely generated free group.

However, it is worth mentioning that the profinite rigidity of a non-primitive surjective word is not known. We now ask a question in the spirit of Conjecture \ref{shalev}.
\begin{qu}
    Is every word in $F_n$, for $n \geq 2$, weakly profinitely rigid?
\end{qu}

In this article, we establish weak profinite rigidity for the powers of surjective words in any finitely generated free group, as well as for the powers of surface words, i.e., $(x_1^2\ldots x_n^2)^d$ in $F_n$ and $([x_1,x_2]\ldots [x_{2n-1},x_{2n}])^d$ in $F_{2n}$ for $n \geq 1$ and any integer $d$. 

 We remark that a positive answer to the above question would affirmatively resolve \cite[Question~18]{Singh} implying that for a chiral word, there exists a finite group on which the word will exhibit chirality. More generally, it would establish that any two words can be distinguished up to endomorphic equivalence if their images as word maps on some finite
group are different.

We now discuss the organization of this article.

In the next section, we generalize the concept of weak profinite rigidity to elements of arbitrary finitely generated groups. We then establish the equivalence of the statement ``Two elements in a finitely generated group $G$ have the same image as word maps on every finite group'' with the assertion that the natural images of these elements in the profinite completion of $G$ are endomorphic equivalent.

In the third section, we examine a word in a free (free profinite, respectively) group that remains fixed under any (continuous, respectively) endomorphism, which then necessarily becomes a (continuous, respectively) automorphism. Such words are known as test words. We demonstrate that the test words in a free group remain test words in its profinite completion, where the free group is identified with its natural image in its profinite completion.

In the final section, we explore weak profinite rigidity among words in free groups. We prove that the powers of surjective words are weakly profinitely rigid. For a given test word $w \in F_n$, we establish that for any word $u \in F_n$ having the same image as $w$ on every finite group, $u$ induces the same probability measure on every finite group as $w$. We also show that the notion of weak profinite rigidity coincides with that of profinite rigidity for test words. As a result, we get that the powers of surface words,  i.e., $(x_1^2\ldots x_n^2)^d$ in $F_n$ and $([x_1,x_2]\ldots [x_{2n-1},x_{2n}])^d$ in $F_{2n}$, for $n \geq 1$ and any integer $d$, are weakly profinitely rigid.

\section{Weak Profinite Rigidity and its equivalent notions}

We generalize the notion of word maps to elements of an arbitrary finitely generated group.

Consider any $n$-tuple $g_1, \ldots,g_n$ in $G$. There is a homomorphism $\phi$ from $F_n$ to $G$ defined by $\phi(x_i) = g_i$ for $1 \leq i \leq n$. Conversely, from the universal property of free groups, any such homomorphism $\phi$ from $F_n$ to $G$ corresponds to a $n$-tuple in $G$. For $w \in F_n,$ we have $w(g_1,\ldots,g_n) = \phi(w).$ In fact $G_w$, the image of the word map $w$ on $G$, is $ \{ \phi(w) \mid \phi \in \mathrm{Hom}(F_n, G)\}$, where $\mathrm{Hom}(F_n,G)$ is the set of all homomorphisms from $F_n$ to $G$. This depiction of word maps will enable us to extend the notion of word maps to elements in any finitely generated groups.

Let $P$ be a finitely generated group and $G$ any group. We use $\mathrm{Hom}(P,G)$ to denote the set of all homomorphisms from $P$ to $G$. For any element $\gamma \in P$, there is a corresponding word map $\gamma$ on $G$. It is a map from $\mathrm{Hom}(P,G)$ to $G$ as follows $\gamma: \phi \mapsto \phi(\gamma)$ where $\phi \in \mathrm{Hom}(P,G)$. We denote the image of $\gamma$ on $G$ by $G_{\gamma}$ and observe that $G_{\gamma} = \{ \phi(\gamma) \mid \phi \in \mathrm{Hom}(P,G)\}.$ For an element $g$ in a finite group $G$, we denote the probability of $g \in G$ induced by an element $\gamma \in P$ by $\mathrm{Pr}_{\gamma}(g)$. In fact, $\mathrm{Pr}_{\gamma}(g)$ can be defined as 

$$\mathrm{Pr}_{\gamma}(g) = \frac{|\phi \in \mathrm{Hom}(P,G) \mid \phi(\gamma) = g |}{|\mathrm{Hom}(P,G)|},$$ which generalizes the probability induced by a word in a free group on a finite group (see \cite{hanany2020some}). Now we prove Lemma \ref{endoequi} which follows as a straightforward corollary of the next lemma.
\begin{lemma}\label{fg}
    Let $P$ be a finitely generated group. Then $\gamma_1, \gamma_2 \in P$ are endomorphic equivalent if and only if  $G_{\gamma_1} = G_{\gamma_2}$ for any group $G.$
\end{lemma}
\begin{proof}
    Suppose there exists an endomorphism $\chi$ of $P$ such that $\chi(\gamma_1) = \gamma_2.$ If $\phi(\gamma_2),$ for some $\phi \in \mathrm{Hom}(P,G),$ is in the image of $\gamma_2$, then $\phi \circ \chi (\gamma_1) = \phi(\gamma_2)$. Hence, we get $G_{\gamma_2} \subseteq G_{\gamma_1}.$ Similarly, we can show  $G_{\gamma_1} \subseteq G_{\gamma_2}.$  For the converse, consider $G = P$. Let $\phi$ be the identity automorphism of $P.$ Then $\gamma_2 = \phi(\gamma_2) \in P_{\gamma_2} = P_{\gamma_1}.$ Therefore, there exists $\chi \in \mathrm{Hom}(P,P)$ with $\chi(\gamma_1) = \gamma_2$. Similarly, there exists $\chi' \in \mathrm{Hom}(P,P)$ such that $\chi'(\gamma_2) = \gamma_1$.
\end{proof}

We generalize Definition \ref{wpr1} of weak profinite rigidity to elements of finitely generated groups.

\begin{definition}\label{def1}
    An element $\gamma$ in a finitely generated group $P$ is said to be weakly profinitely rigid if every element $\beta \in P$ with $G_{\gamma}=G_{\beta}$ for every finite group $G$ is endomorphic equivalent to $\gamma$.
\end{definition}

We aim to study the problem of whether two elements in a finitely generated group (in particular, a free group of finite rank) that have the same image on every finite group will be endomorphic equivalent. To restrict to the class of finite groups, we give the following lemma.

 \begin{lemma}
    Let $\gamma_1, \gamma_2$ be elements in a finitely generated group $P$.  The word maps corresponding to $\gamma_1$ and $\gamma_2$ have the same image on every finite group if and only if they have the same image on every finite quotient of $P$.
 \end{lemma}

\begin{proof}
    We only need to prove the converse, as the forward implication is obvious. Let $G$ be any finite group and $g \in G_{\gamma_1}$,  i.e., there exists $\phi \in \mathrm{Hom}(P,G)$ such that $\phi(\gamma_1) = g$. Now $\phi(P)$ is a finite quotient of $P$ and a subgroup of $G$. Let $i : \phi(P) \longrightarrow G$ be the inclusion. Identify elements of $\phi(P)$ with elements of $G$ via map $i$. Since $g \in \phi(P)_{\gamma_1}$ by restricting the image of $\phi$ to $\phi(P)$, there exists an $\eta \in \mathrm{Hom}(P,\phi(P))$ such that $\eta(\gamma_2) = g$. Hence, $i \circ \eta \in \mathrm{Hom}(P,G)$ such that $i\circ \eta (\gamma_2) =  g$. Therefore, $g \in G_{\gamma_2}$.  
\end{proof}
It suffices to consider the weak profinite rigidity of an element in a finitely generated group only within its finite quotient. Profinite completion is a natural tool for capturing information of a group's finite quotients.

Now, we provide a concise overview of the profinite completion of a group. For any group $P$, the basis for the profinite topology on $P$ consists of the left cosets of subgroups of finite index. The profinite completion of $P$, denoted by $\widehat{P}$, is the inverse limit of its finite quotients. A normal subgroup $N$ of finite index in $P$ is denoted by $N \unlhd_{\mathrm{f.i.}} P.$ Formally, we write
$$\widehat{P} = \varprojlim_{N \unlhd_{\textrm{f.i.}} P} P/N.$$

The group $\widehat{P}$ is equipped with the inverse limit topology, where each finite quotient $P/N$ is regarded as a discrete topological group. With this topology, $\widehat{P}$ is a compact, Hausdorff, and totally disconnected topological group. There is a natural homomorphism $i : P \longrightarrow \widehat{P}$ defined by  $$\gamma = (\gamma N)_{N \unlhd_{\mathrm{f.i.}} P}.$$ 

Let $G$ be a finite group endowed with the discrete topology. By definition, we take homomorphisms from profinite groups to $G$ to be continuous. For our purpose, we don't need to assume continuity as all homomorphisms are continuous when $P$ is finitely generated. Let ${\mathrm{Hom}}(\widehat{P},G)$ be the set of all continuous homomorphisms from $\widehat{P}$ to $G.$ By the universal property of the profinite completion, there is a bijection between $\mathrm{Hom}(P,G)$ and $\mathrm{Hom}(\widehat{P},G).$  To be more precise, for every homomorphism $\phi \in \mathrm{Hom}(P,G)$ there exists a unique  $\widehat{\phi} \in \mathrm{Hom}(\widehat{P},G)$ such that $\phi = \widehat{\phi} \circ i$ \cite[Lemma 3.2.1]{ribes2000profinite}.

Let $\gamma \in P$ and $g \in G$ where $P$ is a finitely generated group and $G$ is a finite group. We define 
\begin{align*}
   \mathrm{Hom}_{\gamma,g}(P,G) &:= \{ \phi \in \mathrm{Hom}(P,G) \mid \phi(\gamma) = g \}, \\
    K_P(G) &:= \bigcap_{N \unlhd_{\mathrm{f.i.}} P, \frac{P}{N} \cong H \leq G} N, \\
    J_P(G) &:= \bigcap_{\phi \in \mathrm{End}(P)} \phi^{-1}(K_P(G)).
\end{align*}

We note that the subgroup $K_P(G)$ was introduced specifically to define $J_P(G)$.
We claim that $J_P(G)$ is a fully invariant subgroup of $P$ of finite index. Since $P$ is finitely generated, there are finitely many subgroups of any given finite index. Consequently, $K_P(G)$ is clearly a normal subgroup of $P$ of finite index. 

For any $\phi \in \mathrm{End}(P)$, consider the map 
$$\chi : P / \phi^{-1}(K_P(G)) \longrightarrow P / K_P(G)$$ 
defined by $\chi(p\phi^{-1}(K_P(G))) = \phi(p)K_P(G).$ It is straightforward to verify $\chi$ is a well-defined injective homomorphism. Therefore, $[P:  \phi^{-1}(K_P(G))] \leq [P : K_P(G)]$, which implies that there are only finitely many possibilities for distinct $\phi^{-1}(K_P(G))$. Thus, $J_P(G)$, being the intersection of finitely many normal subgroups of finite index in $P$, is a normal subgroup of $P$ of finite index. Moreover, the construction of $J_P(G)$ in $P$ ensures that it is fully invariant in $P$.

When we talk about endomorphism in the case of a profinite group, we mean a continuous endomorphism. Now we provide an analogous theorem of \cite[Theorem~2.2]{hanany2020some}.
\begin{theorem}\label{main}
    Let $P$ be a finitely generated group and $\gamma_1, \gamma_2 \in P.$ Then the following are equivalent:
    \begin{enumerate}
        \item $i(\gamma_1)$ and $i(\gamma_2)$ are endomorphic equivalent  in $\widehat{P}$.
        \item For every finite group $G$, and every $g \in G$, the set $\mathrm{Hom}_{\gamma_1,g}(P,G)$ is non-empty if and only if $\mathrm{Hom}_{\gamma_2,g}( P,G)$ is non-empty.
        \item $\gamma_1J$ and $\gamma_2J$ are endomorphic equivalent in $P/J,$ where $J= J_P(G)$ for every finite group $G.$
        \item For every $N \unlhd _\mathrm{f.i.}P,$ there exists $J \unlhd_\mathrm{f.i.} P$ with $J \leq N$ such that $\gamma_1J$ and $ \gamma_2J$ are endomorphic equivalent in $P/J.$
    \end{enumerate}
\end{theorem}

We have that for a group $H$, $\mathrm{End}(H)$ is a monoid and $\mathrm{End}(\widehat{H})$ is a profinite monoid \cite{almeida2020profinite}. More details about profinite monoids can be found in \cite{almeida2020profinite}. We need the following lemma to prove the above theorem. 

\begin{lemma}\label{promon}
    Let $P$ be a finitely generated group. Then $\mathrm{End} (\widehat{P}) = \varprojlim_{J} \mathrm{End}(P/J)$, where the inverse limit is taken over all $J \leq P$ such that $J = J_P(G)$ for some finite group $G.$
\end{lemma}

\begin{proof}
    Let $J_P(G_1) = J_1 \leq J_2 = J_P(G_2), $ and denote $Q_i = P/J_i $ for $i = 1,2.$ Let $\pi: P \to Q_1$ be the quotient map. The image of $J_2$ in $Q_1$ is $J_2/J_1$. We first prove $J_2/J_1 = J_{Q_1}(G_2)$: 

    \begin{itemize}
        \item[($\subseteq$)] For any $x = \pi(j) \in J_2/J_1$ (where $j \in J_2$) and $\psi \in \mathrm{End}(Q_1)$, we need to show $\psi(x) \in K_{Q_1}(G_2)$. Fix a normal subgroup $N' \unlhd Q_1$ with $Q_1/N' \leq G_2$. Define $M = \pi^{-1}(N')$ so $M \unlhd P$ and $P/M \cong Q_1/N' \leq G_2$. Then $K_P(G_2) \subseteq M$. Let $\pi_{N'} : Q_1 \to Q_1/N'$ be the quotient map. Then, the homomorphism $\pi_{N'} \circ \psi \circ \pi: P \to Q_1/N'$ has image in $G_2$, so $K_P(G_2) \subseteq \ker(\pi_{N'} \circ \psi \circ \pi)$. Since $j \in J_2 \subseteq K_P(G_2)$, we have $\psi(\pi(j)) \in N'$. As this holds for all such $N'$ with $Q_1/N' \leq G_2$, $\psi(x) \in K_{Q_1}(G_2)$. Thus, $x \in J_{Q_1}(G_2)$.
        \item[($\supseteq$)] Let $y = \pi(p) \in J_{Q_1}(G_2)$ for some $p \in P$. For any $\phi \in \mathrm{End}(P)$, since $\phi(J_1) \subseteq J_1$, $\phi$ induces an endomorphism $\overline{\phi}$ of $Q_1$ with $\overline{\phi} \circ \pi = \pi \circ \phi$. Then $\overline{\phi}(y) \in K_{Q_1}(G_2)$, so $\phi(p) \in \pi^{-1}(K_{Q_1}(G_2))$. Since $J_1 \subseteq M$ for every $M \unlhd P$ with $P/M \leq G_2$ (as $J_1 \subseteq J_2 \subseteq K_P(G_2) \subseteq M$ for every $M \unlhd P$ with $P/M \leq G_2$), the set $\{M \unlhd P : P/M \leq G_2\}$ equals $\{M \unlhd P : J_1 \subseteq M,  P/M \leq G_2\}$. Hence, we have
        $$
        \pi^{-1}(K_{Q_1}(G_2)) = \bigcap_{\substack{N \unlhd Q_1,\\ {Q_1}/{N} \leq G_2}} \pi^{-1}(N) = \bigcap_{\substack{M \unlhd P, \\ P/M \leq G_2}} M = K_P(G_2).
        $$
        Therefore, $\phi(p) \in K_P(G_2)$. As this holds for all $\phi \in \mathrm{End}(P)$, $p \in J_2$. Therefore, $y \in J_2/J_1$.
    \end{itemize}

    This establishes that the image of $J_2$ in $Q_1$ is equal to $J_{Q_1}(G_2).$ By construction, $J_{Q_1}(G_2)$ is fully invariant in $Q_1$. Consequently, there exists a well-defined monoid homomorphism $\mathrm{End}(Q_1) \to \mathrm{End}(Q_2).$ 

    For $J_1 = J_P(G_1), J_2 = J_P(G_2)$, define $G_3 = P/(J_1 \cap J_2)$. Then $J_3 = J_P(G_3) \leq J_1 \cap J_2 \leq J_1, J_2$. Thus, the subgroups $J = J_P(G)$ (for finite $G$) form a directed set, making $\{\mathrm{End}(P/J)\}_{\{J = J_P(G)\}}$ an inverse system. Hence, $\varprojlim_{J} \mathrm{End}(P/J)$ is well-defined.

    By \cite[Proposition 3.2.2]{ribes2000profinite}, for every $J \unlhd_{\mathrm{f.i.}} P$, we have $P/J \cong \widehat{P}/\overline{i(J)},$ and $\overline{i(J_P(G))} = J_{\widehat{P}}(G).$ Therefore, it suffices to show that $$\mathrm{End} (\widehat{P}) = \varprojlim_{J} \mathrm{End}(\widehat{P}/J),$$
    where the inverse limit is taken over all subgroups $J \leq \widehat{P}$ such that $J = J_{\widehat{P}}(G)$ for some finite group $G.$ Since $J = J_{\widehat{P}}(G)$ is fully invariant in $\widehat{P}$, every endomorphism of $\widehat{P}$ induces an endomorphism of $\widehat{P}/J$ which agrees with the inverse system, leading to a natural continuous homomorphism: $$\omega: \mathrm{End} (\widehat{P}) \longrightarrow \underset{}{}\varprojlim_{J} \mathrm{End}(\widehat{P}/J).$$
    The map $\omega$ is injective because $\underset{G\,\mathrm{finite}}{\cap} J_{\widehat{P}}(G) = \{e_{\widehat{P}}\}.$ The map $\omega$ is surjective because every element of the inverse system $\varprojlim_{J} \mathrm{End}(\widehat{P}/J)$ defines a continuous endomorphism of $\varprojlim_{J} \widehat{P}/J \cong \widehat{P}.$
\end{proof}

\begin{proof}[Proof of Theorem \ref{main}]
\mbox{}
\begin{description}
    \item[$1 \implies 2$] The proof follows from the fact that $\mathrm{Hom}(P,G)$ and $\mathrm{Hom}(\widehat{P},G)$, for every finite group $G$ are in one-one correspondence and we have $| \mathrm{Hom}_{\gamma,g} (P,G)| = |\mathrm{Hom}_{i(\gamma),g}(\widehat{P}, G)|.$ Hence, $$ \vert \mathrm{Hom}_{\gamma_1,g} (P,G)| = |\mathrm{Hom}_{i(\gamma_1),g}(\widehat{P}, G)|  = |\mathrm{Hom}_{i(\gamma_2),g}(\widehat{P}, G)| = |\mathrm{Hom}_{\gamma_2,g} (P,G)|.$$
    \item [$2 \implies 3$] Let $J = J_P(G)$ for some finite group $G$. As $\mathrm{Hom}_{\gamma_1, \gamma_1J}(P, P/J)\neq \emptyset \implies \mathrm{Hom}_{\gamma_2, \gamma_1J}(P,P/J) \neq \emptyset$. Let $f \in \mathrm{Hom}_{\gamma_2, \gamma_1J}(P,P/J).$ Then $J \subseteq \ker(f)$ because of the choice of $J.$ Hence, there exists an induced homomorphism $\Tilde{f}: P/J \longrightarrow P/J$ such that $(\tilde f) (\gamma_2J) = \gamma_1J.$ Similarly, we can show the other way. Hence, $\gamma_1J$ and $\gamma_2J$ are endomorphic equivalent in $P/J$.
    \item [$3 \implies 4$] For every $N \unlhd_{\mathrm{f.i.}}P,$ we can take $J = J_P(P/N).$ 
    \item [$4 \implies 3$] Let $J = J_P(G)$ for some finite group $G.$ By assumption, there exists a subgroup $J' \le J$ such that $J' \unlhd_{\mathrm{f.i.}} P$ with $\gamma_1J'$ and $\gamma_2J'$ endomorphic equivalent. But the image of $J$ in $P/J'$ is $J_{P/J'}(G),$ so this image is fully invariant, hence, every endomorphism of $P/J'$ induces an endomorphism in $P/J$. Hence, we conclude that $\gamma_1J, \gamma_2J$ are endomorphic equivalent in $P/J.$
    \item [$3 \implies 1$] For each $J = J_P(G)$, define
$$
S_J = \left\{ \alpha \in \mathrm{End}(P/J) \mid \alpha(\gamma_1 J) = \gamma_2 J \right\}.
$$
Each $S_J$ is non-empty by hypothesis and closed in the finite monoid $\mathrm{End}(P/J)$ (endowed with the discrete topology). The canonical projections $$\pi_J: \varprojlim_{J} \mathrm{End}(P/J) \to \mathrm{End}(P/J)$$ are continuous. Thus, $\pi_J^{-1}(S_J)$ is closed in $\varprojlim_{J} \mathrm{End}(P/J)$ for each $J$.

For any finite collection $\{J_1, \dots, J_n\}$, directedness yields $J \subseteq \bigcap_{i=1}^n J_i$ with $J = J_P(G)$. For $\alpha \in S_J$, full invariance of $J_i/J$ in $P/J$ ensures $\alpha$ induces $\alpha_{J_i} \in S_{J_i}$ which is also the image of $\alpha$ in $\mathrm{End}(P/J_i)$. In other words, an element of $\varprojlim_{J} \mathrm{End}(P/J)$ defined by $\alpha$ at $J^{\text{th}}$ coordinate forces the $J_i^{\text{th}}$ coordinate to be $\alpha_{J_i}$ for each $1 \leq i \leq n$. Hence, the compatibility condition implies that $\pi_{J}^{-1}(\alpha) \subseteq \pi_{J_i}^{-1}(\alpha_{J_i})$ for each $1 \leq i \leq n$. Thus, $\bigcap_{i=1}^n \pi_{J_i}^{-1}(S_{J_i}) \neq \emptyset$. So the family $\{\pi_J^{-1}(S_J)\}_J$ has the finite‐intersection property. Compactness of $\varprojlim_{J} \mathrm{End}(P/J)$ then gives:
$$
\bigcap_J \pi_J^{-1}(S_J) \neq \emptyset.
$$
Take $(\alpha_J)_J$ in this intersection. Then $\alpha_J(\gamma_1 J) = \gamma_2 J$ for all $J$. By Lemma \ref{promon}, there exists $\widehat{\alpha} \in \mathrm{End}(\widehat{P})$ mapping $i(\gamma_1) = (\gamma_1 J)_J$ to $i(\gamma_2) = (\gamma_2 J)_J$. Repeating with $\gamma_1,\gamma_2$ swapped yields $\widehat{\beta} \in \mathrm{End}(\widehat{P})$ mapping $i(\gamma_2)$ to $i(\gamma_1)$.
\end{description} 
\end{proof}

The above theorem is the reason to generalize Definition \ref{def1}.

\begin{definition}
    Let $P$ be a finitely generated group. An element $\gamma_1 \in P$ is said to be weakly profinitely rigid if it is endomorphic equivalent to every $\gamma_2\in P$ for which $i(\gamma_1)$ and $i(\gamma_2)$ are endomorphic equivalent in $\widehat{P}$.
\end{definition}

 It is straightforward that every element in a finite group is weakly profinitely rigid, and also profinitely rigid. But two words which are endomorphic equivalent need not be automorphic. For example, take two elements of order $2$ in group $S_3 \times \Z_2$, namely $((1\,2),\bar{0})$ and $((1 ),\bar{1})$, they are endomorphic to each other but they are not automorphic. 
 
 The study whether endomorphic equivalence is the same as automorphic equivalence has been initiated by Calvert, Dutta and Prasad in \cite{calvert2013degeneration} for countable abelian groups. For $a$ and $b$, endomorphic equivalent elements in a group $G$, they called $a$ and $b$ degenerate to each other. They proved the following result:

 \begin{lemma}\cite{calvert2013degeneration}
     Let $G$ be a countable abelian group. Two endomorphic equivalent elements lie in the same $\mathrm{Aut}(G)$-orbit. 
 \end{lemma}

\begin{theorem}
    Every element of a finitely generated abelian group is weakly profinitely rigid.
\end{theorem}
\begin{proof}
    Let $P$ be a finitely generated abelian group. Suppose $a, b \in P$ have identical images in every finite group as word maps. By Theorem \ref{main}, we get that $aJ$ and $bJ$ are endomorphic equivalent for each $J = J_P(G)$, where $G$ is a finite group. Since quotient groups $P/J$ are abelian, $aJ$ and $bJ$ are automorphic by \cite{calvert2013degeneration}. Hence, $i(a)$ and $i(b)$ are automorphic in $\widehat{P}$. From \cite[Theorem~2.2]{hanany2020some}, we get that $a$ and $b$ induce the same probability measure on every finite group $G$. Therefore, any two elements of a finitely generated abelian group having the same image on every finite group induce the same probability measure on every finite group. Since every element in a finitely generated abelian group is profinitely rigid \cite{hanany2020some}, hence, $a$ and $b$ are endomorphic equivalent. This proves our result.
\end{proof}
 
We present an example of a group in which no element is weakly profinitely rigid.
\begin{ex}
    Let $S$ be a finitely generated infinite simple group. Then no element of $S$ is weakly profinitely rigid, since $\mathrm{Hom}(S,G)$, for any finite group $G,$ consists of only the trivial homomorphism.
\end{ex}

 \begin{rem}
     In fact, non-trivial elements from the intersection of all normal subgroups of finite index will never be weakly profinitely rigid. Hence, every non-residually finite group contains an element which is not weakly profinitely rigid. 
 \end{rem}

All known examples of groups containing elements that are not weakly profinitely rigid are non-residually finite. This leads us to the following intriguing question.

\begin{qu}
    Is every element in a residually finite group weakly profinitely rigid?
\end{qu}

\section{Test Elements in Free Profinite Group}

An element $g \in G$ is said to be a test element if any endomorphism $\phi$ of the group $G$ that satisfies $\phi(g) = g$ is an automorphism. In this section, we will discuss test words in a finitely generated free group that remain test words in its profinite completion.

\begin{definition}
    A subgroup of a group $G$ is said to be a retract of $G$ if there exists a homomorphism $\phi: G \longrightarrow H$ such that $\phi\vert_H$ is the identity map. 
\end{definition}
Turner has classified test words in free groups.

\begin{theorem}\cite{Turner}\label{turner1}
    The test words in a free group $F_n$ are words that are not contained in any proper retract.
\end{theorem}
Snopce and Tanushevski have proved the analogous result for a finitely generated profinite group.

\begin{theorem}\cite[Theorem~3.5]{snopce2017test}\label{snopce}
The test elements of a finitely generated profinite group are exactly the elements that are not contained in any proper retract.    
\end{theorem}

\begin{pro}\label{prof}
    A word $w \in F_n$ is a test element if and only if $i(w)$ is a test element in $\widehat{F_n}$.  
\end{pro}

\begin{proof}
    Suppose $w$ is not a test element in $F_n$. By Theorem \ref{turner1}, there exists a proper retract $R$ of $F_n$ containing $w$. Let $r: F_n \longrightarrow R$ be the retraction. Taking the induced map on its profinite completion, we get $\widehat{r} : \widehat{F_n} \longrightarrow \widehat{R}$, which will again be a proper retraction, since the rank of a proper retract of $F_n$ is stricly less than $n$ \cite{Turner}. Hence, $i(w)$ is not a test word in $\widehat{F_n}$ by Theorem \ref{snopce}.

    Suppose $i(w)$ is not a test word. By Theorem \ref{snopce}, there exists a proper retract $S$ of $\widehat{F_n}$ containing $i(w)$. Consider $R = S \cap F_n$, where $F_n$ and $i(F_n)$ are identified. $R$ is clearly a proper retract of $F_n$ containing $w$. Hence, $w$ is not a test word.
\end{proof}

The next lemma is the explicit description of test words in $F_2$ given by Turner.
\begin{lemma}\cite{Turner}\label{turner2}
     Let $w \in F_2 \setminus \{e\}$. Write $w = u^m, m \in \Z \setminus \{0\}$ such that $u$ is not a proper power word. Suppose $u = x_1^{m_1}x_2^{m_2}c,$ where $c \in F_2'$. Then $w$ is a test word if and only if $\gcd(m_1,m_2) \ne 1$.
\end{lemma}

\section{Main Results}
In this section, we establish that the powers of surjective words in $F_n$ and surface words,  i.e., $(x_1^2\ldots x_n^2)^d$ in $F_n$ and $([x_1,x_2]\ldots [x_{2n-1},x_{2n}])^d$ in $F_{2n}$, for $n \geq 1$ and any integer $d$, are weakly profinitely rigid. 
\begin{theorem}\label{main1}
    Let $F_n$ be a free group. If $w \in F_n$ is weakly profinitely rigid, then $w^d, d \in \Z$, is also weakly profinitely rigid.
\end{theorem}
\begin{lemma}\cite[Theorem~4.8]{hanany2020some}\label{L1}
    Let $w \in F_n$ such that every image of $w$ on every symmetric group $S_r$ is $d^{\text{th}}$ power for some $d \in \mathbb{Z}$, then there exists $v \in F_n$ such that $w = v^d$. 
\end{lemma}
  
\begin{lemma}\cite[Corollary~4.4]{hanany2020some}\label{L2}
    Every root of $w \in F_n$ in $\widehat{F_n}$ belongs to $F_n$.
\end{lemma}

\begin{proof}[Proof of Theorem \ref{main1}]
     Lemma \ref{L1} implies that any word having the same image as $w^d$ on every finite group must be of the form $v^d$ for some $v \in F_n$. By Theorem \ref{main}, there exist endomorphisms $\phi_1, \phi_2 \in \mathrm{End}(\widehat{F_n})$ such that $\phi_1(i(w^d)) = i(v^d)$ and $\phi_2(i(v^d)) = i(w^d)$. Applying Lemma \ref{L2}, we conclude that every $d^{\text{th}}$ root of $\phi_1(i(w))^d$ belongs to $F_n$, which gives $\phi_1(i(w)) = i(v).$ Similarly, we obtain $\phi_2(i(v)) = i(w)$. Since $w$ is weakly profinitely rigid, it follows that $w$ and $v$ are endomorphic equivalent.Consequently, $w^d$ and $v^d$ are also endomorphic equivalent, proving that $w^d$ is weakly profinitely rigid.
\end{proof}

\begin{theorem}\label{prim}
    Every surjective word in any finitely generated free group is weakly profinitely rigid.
\end{theorem}

\begin{proof}
    Let $w \in F_n$. Express $w$ as $w = x_1^{t_1}x_2^{t_2}\ldots x_n^{t_n}c$, where $c$ belongs to the commutator subgroup of $F_n$. We first prove that $w$ is surjective as a word map on every finite group if and only if $\gcd\{t_1,t_2,\ldots,t_n\}=1$. Suppose, for contradiction, that $\gcd\{t_1,t_2,\ldots,t_n\} = r \neq 1$.  If $r \geq 2,$ then for $G = \Z_r$, we get $G_w$ to be trivial, which is not surjective. Similarly, if $t_i = 0$ for $i \in \{1,2,\ldots,n\}$, then for any finite abelian group $G$, the image $G_w$ will be trivial. If $r = 1$, then there exists $d_i \in \Z$ such that $\sum d_it_i = 1$. For any $g \in G$, defining $x_i \mapsto g^{n_i}$ ensures that $w(g^{n_1}, g^{n_2},\ldots,g^{d_n}) = g,$ which implies $G_w = G,$ proving that $w$ is surjective. 
    
    The result now follows from the observation after Definition \ref{def_end} that any two surjective words are endomorphic equivalent.
\end{proof}

Every primitive word is both weakly profinitely rigid and profinitely rigid \cite{puder2015measure}. However, the profinite rigidity of non-primitive surjective words, such as $x_1^2x_2x_1^{-1}x_2^{-1} \in F_2,$ remains an open question.
\begin{cor}
   Let $w \in F_n$ be any surjective word. Then for any integer $d$, the power $w^d$ is weakly profinitely rigid.  
\end{cor}
\begin{proof}
    This follows directly from Theorem \ref{prim} and Theorem \ref{main1}.
\end{proof}

\begin{theorem}\label{main2}
    Let $w_1 \in F_n$ be a test word. If another $w_2 \in F_n$ has the same image as $w_1$ on every finite group, then $w_1$  and $w_2$ induce the same probability measure on every finite group. 
\end{theorem}
\begin{proof}
    By Theorem \ref{main}, we deduce that $i(w_1)$ and $i(w_2)$ are endomorphic equivalent. Let $\sigma_1, \sigma_2 \in \mathrm{End}(\widehat{F_n})$ such that $\sigma_1(i(w_1)) = i(w_2)$ and $\sigma_2(i(w_2)) = i(w_1)$. Then $\sigma_2 \circ \sigma_1$ fixes $i(w_1)$. By Proposition \ref{prof}, it follows that $\sigma_2 \circ \sigma_1$ is an automorphism of $\widehat{F_n}$. It follows that $\sigma_2$ is surjective. Since finitely generated profinite groups are hopfian (see \cite[Proposition~2.5.1]{ribes2000profinite}), it follows that $\sigma_2 $ is an automorphisms $\widehat{F_n}$. Consequently, $\sigma_1 = \sigma_2^{-1} \circ (\sigma_2 \circ \sigma_1)$ is also an automorphism. Therefore, $i(w_1)$ and $i(w_2)$ are automorphic in $\widehat{F_n}$. Hence, by \cite[Theorem~2.2]{hanany2020some}, $w_1$ and $w_2$ induce the same probability measure on every finite group. 
\end{proof}

\begin{theorem}\label{theorem}
     A test word in $F_n$ is weakly profinitely rigid if and only if it is profinitely rigid. 
\end{theorem}
\begin{proof}
    Suppose $w_1$ is a test word in $F_n$ that is weakly profinitely rigid. Let $w_2 \in F_n$ be another word that induces the same probability measure as the word $w_1$ on every finite group. Clearly, $w_1$ and $w_2$ have the same image as word maps on every finite group. Since $w_1$ is weakly profinitely rigid, it follows that $w_1$ and $w_2$ are endomorphic equivalent in $F_n$. Let $\sigma_1$ and $\sigma_2$ be two endomorphism of $F_n$ such that $\sigma_1(w_1) = w_2$ and $\sigma_2(w_2) = w_1$. The composition $\sigma_2 \circ \sigma_1$ fixes $w_1$, and since $w_1$ is a test word $w_1$, $\sigma_2 \circ \sigma_1$ must be an automorphism. Consequently, $\sigma_2$ is surjective. Because free groups are hopfian, $\sigma_2$ must be an automorphism, implying that $\sigma_1$ is also an automorphism.  Hence, $w_1$ is profinitely rigid. 

    Conversely, suppose $w_1 \in F_n$ is a profinitely rigid test word. Let $w_2 \in F_n$ be another word such that it has the same image as the word $w_1$ on every finite group. By Theorem \ref{main2}, $w_1$ and $w_2$ induce the same probability measure on every finite group. Profinite rigidity of $w_1$ implies that $w_1$ and $w_2$ are automorphic in $F_n$. Hence, they are endomorphic equivalent. Thus, $w_1$ is weakly profinitely rigid. 
\end{proof}

\begin{theorem}\label{thm}
    Powers of surface words, i.e., $(x_1^2\ldots x_n^2)^d$ in $F_n$ and $([x_1,x_2]\ldots [x_{2n-1},x_{2n}])^d$ in $F_{2n}$ for any integer $d$, are weakly profinitely rigid.
\end{theorem}

As pointed out in \cite[pp.~ 93-94]{puder2015measure}, the profinite rigidity of a word $w \in F_n $ is equivalent to the orbit $\mathrm{Aut}(F_n)w$ being closed in the profinite topology. The next lemma follows from \cite[Corollary~4]{wilton2021profinite} which states that the automorphism orbit of a surface word is closed in the profinite topology and \cite[Remark~6]{wilton2021profinite}.

\begin{lemma}\cite{wilton2021profinite}\label{lem1}
   Powers of surface words, $(x_1^2\ldots x_n^2)^d$ in $F_n$ and $([x_1,x_2]\ldots [x_{2n-1},x_{2n}])^d$ in $F_{2n}$ for any integer $d$, are profinitely rigid.
\end{lemma}

\begin{proof}[Proof of Theorem \ref{thm}]
    H. Zieschang \cite{zieschang1966automorphismen} showed that surface words, i.e., $x_1^2\ldots x_n^2$ in $F_n$ and $[x_1,x_2]\ldots [x_{2n-1},x_{2n}]$ in $F_{2n}$, are test words. Since the power of a test word in a free group is also a test word in the free group, the result follows directly from Lemma \ref{lem1} and Theorem \ref{theorem}.
\end{proof}

 By Lemma \ref{turner2}, every word in $F_2$ is either a power of a surjective word or a test word.  The following corollary can be seen as another characterization of test words in $F_2$.

\begin{cor}
    Let $w \in F_2 \setminus \{e\}$ be a word. If, for every word $u \in F_2$, the condition that $w$ and $u$ have the same image on every finite group implies that they induce the same probability measure on every finite group, then $w$ is a test word. 
\end{cor}
\begin{proof}
     It suffices to consider non-trivial non-test words by Theorem \ref{main2}. We shall show that for a non-trivial non-test word $w$ in $F_2$, we get another word $u$ in $F_2$ such that $w$ and $u$ have the same image on every finite group but induce different probability measures on some finite group.
     
     If $w$ is a non-trivial non-test word in $F_2$, then it is a power of some surjective word $w_1$, say $w = w_1^d$. A primitive word can never be automorphic to a non-primitive surjective word. Hence, we have a surjective word $u_1$ that is not automorphic to $w_1$, such that $w = w_1^d$ and $u=u_1^d$ will have the same image on every finite group. We take $u_1$ such that only one of $w_1$ and $u_1$ is primitive. Using \cite{hanany2020some}, there exists a finite group $G$ on which $w$ and $u$ will not induce the same probability measure.
\end{proof}
\section*{Acknowledgement}
The author is thankful to Prof Doron Puder for his interest in this work and for Lemma \ref{lem1}. The author also thanks to Prof Amber Habib, Prof Krishnendu Gongopadhyay and Dr Sumit Chandra Mishra for their careful reviews and suggestions to improve the article. The author is grateful to the anonymous reviewer for their valuable feedback, which enhanced the readability of this article.

\bibliographystyle{siam}
\bibliography{bib}
\end{document}